\documentclass[twocolumn]{ceurart}

\usepackage{listings}
\usepackage{amsmath, amssymb, amsfonts, amsfonts, amsthm,latexsym}
\usepackage[mathscr]{euscript}
\usepackage{wasysym}
\usepackage{graphics}
\usepackage{enumerate}
\usepackage{enumitem}
\usepackage{comment}
\usepackage[normalem]{ulem}
\usepackage{graphicx}
\usepackage{xcolor}
\usepackage{bm}
\usepackage{listings}
\usepackage[utf8]{inputenc}
\usepackage[T1]{fontenc}
\newtheorem{theorem}{Theorem}[section]

\newtheorem{corollary}[theorem]{Corollary}

\newtheorem{proposition}[theorem]{Proposition}

\newtheorem{remark}[theorem]{Remark}
\newtheorem{example}[theorem]{Example}

\newcommand{\dG}{\overrightarrow{G}}

\begin{document}

\copyrightyear{2022}
\copyrightclause{Copyright for this paper by its authors.
  Use permitted under Creative Commons License Attribution 4.0
  International (CC BY 4.0).}

\conference{ITAT'24: Information technologies -- Applications and Theory, September 20--24, 2024, Drienica, Slovakia}

\title{Forbidden paths and cycles in the undirected underlying graph of a 2-quasi best match graph}
\author{Annachiara Korchmaros}[
orcid=0000-0002-7334-669X,
email=annachiara.korchmaros@uni-leipzig.de,
url=https://akorchmaros.com/,
]

\address{Bioinformatics Group, Department of Computer Science \&
  Interdisciplinary Center for Bioinformatics, Universit{\"a}t Leipzig,
  H{\"a}rtelstra{\ss}e~16--18, D-04107 Leipzig, Germany.}

\begin{abstract} 
The undirected underlying graph of a 2-quasi best match graph (2-qBMG) is proven not to contain any induced graph isomorphic to $P_6$ or $C_6$. This new feature allows for the investigation of 2-BMGs further by exploiting the numerous known results on $P_6$ and $C_6$ free graphs together with the available polynomial algorithms developed for their studies. In this direction, there are also some new contributions about dominating bicliques and certain vertex decompositions of the undirected underlying graph of a 2-qBMG.  
\end{abstract}


\maketitle

\section{Introduction}
Quasi-best match graphs are directed vertex-colored graphs explained by phylogenetic trees.  
A tree-free characterization is known so far for the case of two colors~\cite{korchmaros2023quasi}:   
A {\emph{two-color quasi best match graph}} (2-qBMG) is a bipartite directed graph $\overrightarrow{G}$ without loops and parallel edges which has the following three properties:
\begin{itemize}
\item[(N1)] if $u$ and $v$ are two independent vertices then there exist no vertices $w,t$ such that $ut,vw,tw$ are edges;
\item[(N2)] bi-transitive, i.e., if $uv,vw,wt$ are edges then $ut$ is also an edge;
\item[(N3)] if $u$ and $v$ have common out-neighbor then either all out-neighbors of $u$ are also out-neighbors of $v$ or
all out-neighbors of $v$ are also out-neighbors of $u$.
\end{itemize}
A 2-qBMG is a \emph{two-colored best match graph} (2-BMG) if it is sink-free, i.e., no vertex has empty out-neighborhood, and it is a \emph{reciprocal two-best match graph} (reciprocal 2-BMG) if all its edges are symmetric. 
Moreover, 2-qBMGs but not 2-BMGs form a hereditary class~\cite{korchmaros2023quasi}.

2-BMGs, 2-qBMGs, and more generally best match graphs 
have been the subjects of intensive studies, also motivated by their relevant roles in the current investigation of orthology predictions from sequence similarity in the case that the gene families histories are free from gene transfers; see~\cite{ramirez2024revolutionh, schaller2021corrigendum,korchmaros2021structure, korchmaros2023quasi, schaller2021heuristic, schaller2021complexity, schaller2021complete, hellmuth2024theory}. A major contribution ~\cite[Theorem 3.4]{schaller2021complexity} is a characterization of 2-qBMGs by three types of forbidden induced subgraphs, called of types $(F_1),(F_2)$ and $(F_3)$; see Section~\ref{bg}.        

This investigation focuses on specific properties of the underlying undirected graph of a 2-qBMG. Actually, such undirected graphs have been touched so far only marginally in the study of 2-qBMGs except in the very special case of reciprocal 2-BMGs; see~\cite{manuela2020reciprocal,hellmuth2020complexity}. 
The main motivation for this study is the richness of the literature on undirected graphs. It is indeed much richer than on directed graphs, with plenty of fundamental works on forbidden configurations, decompositions, and vertex-colorings, as well as on algorithms and valuations on computational complexity. The aim is to exploit some of these deeper results on undirected graphs to gain new insights into 2-qBMGs. It may happen, however, that going back to a digraph $\overrightarrow{G}$ from its underlying undirected graph $G$, a deep result on $G$ only has a very limited, or even trivial impact on $\overrightarrow{G}$. Thus, the structural relationship between $\overrightarrow{G}$ and $G$ is far to be immediate. 
For instance, \cite[Theorem 3.4]{schaller2021complexity} does not yield an analog characterization in terms of forbidden induced subgraphs of the underlying undirected graph. 
This depends on the fact that the underlying undirected graph of a digraph of type $(F_i)$, $1\le i \le 3$ with only required edges is either a path of length $4$ or $5$ but it is not necessarily a forbidden subgraph of the underlying undirected graph of a 2-qBMG. 

Therefore, the challenge is to find appropriate results and knowledge on undirected graphs that may produce relevant contributions to the study of 2-qBMGs via their underlying undirected graphs.

The main results in this direction are Theorem~\ref{thm:P6free} and Theorem~\ref{thm:C6free}, which state that the underlying undirected graph of every 2-qBMG is $P_6$-free and $C_6$-free, that is, both $P_6$ and $C_6$ are forbidden induced subgraphs. This result is sharp as the underlying undirected graph of a 2-qBMG may have an induced path $P_5$ (and hence induced $P_4$ as well); see Theorem~\ref{thm:P5}. Therefore, the underlying undirected graphs of a type $(F_i)$ digraph, $1\le i \le 3$, are not forbidden subgraphs for the underlying undirected graph of 2-qBMGs.

In their seminal paper~\cite{bacsotuza1} appeared in 1993, Bacs\'o and Tuza pointed out the strong relationship between $P_k$-freeness and dominating subsets in undirected graphs. Relevant contributions on such a relationship were given in several papers, especially in~\cite{bacsotuza2,brandstadt2006p6,van2010new,liu1994dominating,liu2007characterization}. In particular, for any bipartite graph $G$,  Liu and Zhou proved that $G$ is both $P_6$-free and $C_6$-free if and only if every connected induced subgraph of $G$ has a dominating biclique; see~\cite[Theorem 1]{liu1994dominating}. An independent constructive proof of the Liu-Zhou theorem is due to P. van't Hof and D. Paulusma~\cite{van2010new} where the authors also provide an algorithm that finds such a dominating biclique of a connected $P_6$-free graph in polynomial time. For a discussion on the computational complexity of the problem of finding (not necessarily dominating) bicliques in undirected bipartite graphs, see~\cite{peeters2003maximum}. An upper bound on the number of bicliques of $C_6$-free undirected bipartite graphs in terms of the sizes of their color classes $U,W$ is given in~\cite[Theorem 3.3]{prisner2002} where it is shown that such number does not exceed $|U|^2 |W|^2$.   

Recognition and optimization problems for both $P_6$ and $C_6$-free bipartite undirected graphs and certain decompositions involving $K\oplus S$ graphs have been studied by several authors, following the paper \cite{fouquet1999bipartite}. In particular, it is shown that the class of both $P_6$ and $C_6$-free bipartite graphs can be recognized in linear time. Also, efficient solutions for two NP-hard problems are presented in this class of graphs: the maximum balanced biclique problem and the maximum independent set problem. For more details, see the recent paper~\cite{quaddoura2024bipartite}. 

A contribution to the study of the vertex decomposition problem for 2-qBMGs in smaller connected 2-qBMGs of type (A) is given in Theorem~\ref{thm:connected-typeA}, where a connected 2-qBMG is of type (A) if its underlying undirected graph is a $K\oplus S$ graph, that is, it has a vertex decomposition into a (dominating) biclique $K$ and a stable set $S$, that is, any two vertices in $S$ are independent.

\section{Background}\label{bg}

The notation and terminology for undirected graphs are standard. Let $G$ be an undirected graph with vertex-set $V(G)$ and edge-set $E(G)$. 
In particular, a path $P_n$ of an undirected graph $G$ is a sequence $v_1v_2\cdots v_n$ of pairwise distinct vertices of $G$ such that $v_iv_{i+1}\in E(G)$ for $i=1,\ldots,n-1$. A $P_n$ \emph{path-graph} is an undirected graph on $n$ vertices with a path $v_1v_2\cdots v_n$ containing no chord, i.e. any edge in $E(G)$ is $v_iv_{i+1}\in E(G)$ for some $1\le i < n$. A $P_n$ path-graph with a path $v_1v_2\cdots v_n$ containing no chord is a bipartite graph with (uniquely determined) color sets $\{v_1,v_3,\ldots\}$ and $\{v_2,v_4,\ldots\}$. An undirected graph is  $P_n$-\emph{free} if it has no induced subgraph isomorphic to a $P_n$ path-graph. A {\emph{cograph}} is a $P_4$-free graph.  

A cycle $C_n$ of an undirected graph $G$ is a sequence $v_1v_2\cdots v_nv_1$ of pairwise distinct vertices of $G$ such that $v_iv_{i+1}\in E(G)$ for $i=1,\ldots,n-1$, and $v_nv_1\in E(G)$.
A $C_n$ \emph{cycle-graph} is an undirected graph on $n$ vertices with a cycle $v_1v_2\cdots v_nv_1$ containing no chord, i.e. any edge in $E(G)$ is either $v_iv_{i+1}\in E(G)$ for some $1\le i \le n-1$, or $v_nv_1$. For $n$ even, a $C_n$ cycle-graph with a cycle $v_1v_2\cdots v_nv_1$ containing no chord is a bipartite graph with (uniquely determined) color classes $\{v_1,v_3,\ldots\}$ and $\{v_2,v_4, \ldots\}$. Bipartite cycle-graphs can only exist for even $n$. An undirected graph is $C_n$-\emph{free} if it has no induced subgraph isomorphic to an $C_n$ cycle-graph. 

A set $D\subseteq V(G)$ is a \emph{dominating set} of an undirected graph $G$ if, for any vertex $u\in V(G)\setminus D$, there is a vertex $v\in D$ such that $uv\in E(G)$. We also say that $D$ dominates $G$. A subgraph $H$ of $G$ is a \emph{dominating subgraph} of $G$ if the vertex set of $H$  dominates $G$.

An undirected graph $G$ is of type $K\oplus S$ if either $G$ is degenerate, i.e. if it has an isolated vertex, or there is a partition of $V(G)$ into two sets: a biclique set $K$ and a stable set $S$; see~\cite{fouquet1999bipartite,quaddoura2024bipartite}. 

For directed graphs $\overrightarrow{G}$ with vertex-set $V(\dG)$ and edge-set $E(\dG)$, in particular for 2-qBMGs, notation and terminology come from~\cite{korchmaros2023quasi}. In particular, $N^+(v)$ and $N^-(v)$ stand for the set of out-neighbours and in-neighbours of $v$ in $\dG$,  
and $v$ is a {\emph{sink}} when $N^+(v)=\emptyset$, and $v$ is a source when $N^-(v)=\emptyset$. 
A digraph $\dG$ is \emph{oriented} if $uv\in E(\dG)$ implies $vu\notin E(\dG)$ for any $u,v\in V(\dG)$. An oriented digraph has a {\emph{topological vertex ordering}} if its vertices can be labeled with $v_1, v_2, \ldots$ such that for any $v_iv_j\in E(\dG)$ we have $i<j$. A sufficient condition for an oriented digraph to have a topological ordering is to be \emph{acyclic}, that is, there is no directed cycle in the digraph. An {\emph{orientation}} of a digraph $\dG$ is a digraph obtained from $\dG$ by keeping the same vertex set but retaining exactly one edge from each symmetric edge. From \cite[Lemma 2.2]{korchmaros2021structure} and \cite[Theorem 3.8]{korchmaros2021structure}, any orientation of a 2-qBMG is acyclic whenever at least one of the following conditions are satisfied:
\begin{itemize}
\item[($*$)] no two (or more than two) symmetric edges of $\overrightarrow{G}$ have a common endpoint;
\item[($**$)] no two (or more than two) vertices of $\overrightarrow{G}$ are equivalent, i.e. no two vertices have the same in- and out-neighbors.
\end{itemize}
Acyclic-oriented digraphs are odd–even graphs. For a pair $(\mathfrak{A}, \mathfrak{O})$ where $\mathfrak{A}$ is a finite set of non-negative even integers and $\mathfrak{O}$ is a set of positive odd integers, the associated odd–even oriented digraph $\dG$ has vertex-set $\mathfrak{A}$ and edge-set with $ab\in E(\dG)$ when both $\frac{1}{2}
(a+b)$ and $\frac{1}{2}(b-a)$ belong to $\mathfrak{O}$. This oriented bipartite graph has color classes $U=\{a:a \equiv 0 \pmod 4, a \in \mathfrak{A}\}$ and $W=\{a:a \equiv 2 \pmod 4, a \in \mathfrak{A}\}$. 
An oriented bipartite digraph is a {\emph{bitournament}} if for any two vertices $u,v \in V$ with different colors, either $uv \in E(\dG)$, or $uv \in E(\dG)$. A bi-transitive bitournament is an odd-even digraph; see \cite[Proposition 3.9]{korchmaros2021structure}.

In addition, for four vertices $x_1,x_2,x_3,y$ of $\dG$, we say that $[x_1,x_2,x_3,y]$ is an (N1)-configuration if  $x_1x_2,x_2x_3,$ $yx_3 \in E(\dG)$ but either $x_1y\in E(\dG)$ or $yx_1\in E(\dG)$; in other words when condition (N1) holds for $u=x_1,t=x_2,w=x_3,v=y$ and therefore $x_1$ and $y$ are not independent. 

A 2-qBMG is {\emph{degenerate}} if it has an isolated vertex. We stress that a non-degenerate 2-qBMG may be {\emph{trivial}}, as it may be the union of pairwise disjoint (possible symmetric) edges. If this is the case, then (N1), (N2), and (N3) trivially hold in the sense that none of the conditions required in (N1), (N2), and (N3) is satisfied. Also, a directed graph is said to be $P_n$-free or $C_n$-free if its undirected underlying graph is $P_n$-free or $C_n$-free,  respectively. 

Let $\overrightarrow{G}_4$ denote a bipartite digraph on four vertices where $V(\overrightarrow{G}_4) =\{x_1, x_2, y_1, y_2\}$ and the color classes are $\{x_1,x_2\}$ and $\{y_1,y_2\}$. Then $\dG_4$ is of {\emph{type}} $(F_1)$ {\emph{with required edges}} if $E(\dG_4)=\{x_1y_1,y_2x_2,y_1x_2\}$.whereas $\dG_4$ is of {\emph{type}} $(F_2)$ {\emph{with required edges}} if $E(\dG_4)=\{x_1y_1,y_1x_2,x_2y_2\}$. The undirected underlying graph $G_4$ of a $\overrightarrow{G}_4$ of type either $(F_1)$ or $(F_2)$ is a path of length $4$.

Let $\overrightarrow{G}_5$ denote a bipartite graph on five vertices where $V_5(\dG_5) =\{x_1, x_2, y_1, y_2,y_3\}$ and the color classes are $\{x_1,x_2\}$ and $\{y_1,y_2,y_3\}$. Then $\dG_5$ is of {\emph{type}} $(F_3)$ if {\emph{with required edges}} $E(\dG_5)=\{x_1y_1,x_2y_2,x_1y_3,x_2y_3\}$. The undirected underlying graph $G_5$ of $\overrightarrow{G}_5$ of type $(F_3)$ is a path of length $5$. 

Let $\dG_1$ and $\dG_2$ be two directed graphs on the same vertex set $V$. Then $\dG_1\cong \dG_2$, that is, $\dG_1$ and $\dG_2$ are isomorphic, if there exists a edge-preserving permutation $\varphi$ on $V$, i.e. $v_1v_2\in E(\dG_1)$ if and only if $\varphi(v_1)\varphi(v_2)\in E(\dG_2)$.  

A tree $T$ is \emph{phylogenetic} if every node is either a leaf or has at least two children. $T$ is a \emph{rooted} tree if one of the nodes is chosen as \emph{root} denoted by $\rho_T$. In a rooted tree, its root is typically drawn as the top node, and the edges are directed from the parent nodes to their child nodes. In this contribution, all trees are rooted and phylogenetic.
Let $(T,\sigma)$ be a leaf-colored tree with leaf set $L$, set of colors $S$, leaf-coloring surjective map $\sigma: L\rightarrow S$, and rooted at $\rho_T$. For any two leaves $x,y\in V$, $lca(x,y)$ denotes the last common ancestry between $x$ and $y$ on $T$. $T$ is \emph{phylogentic} if all nodes have at least two children except the leaves.
A leaf $y \in L$ is a best match of the leaf $x\in L$ if $\sigma(x)\ne\sigma(y)$ and $lca(x,y) \preceq lca(x,z)$, i.e. $lca(x,z)$ is an ancestor of $lca(x,y)$, holds for all leaves $z$ of color $\sigma(y) =\sigma(z)$. The {\emph{BMG (best match graph) explained by}} $(T,\sigma)$ is the directed graph whose vertices are the leaves of $T$ where $xy\in E(\dG)$ if $y$ is a best match of $x$ is a vertex-colored digraph with color set $S$. 

A truncation map $u_T : L\times S \rightarrow T$ assigns to every leaf $x \in L$ and color $\sigma \in S$ a vertex of T such that $u_T(x,s)$ lies along the unique path from $\rho_T$ to $x$ and that $u_T(x,\sigma(x)) = x$. A leaf $y\in  L$ with color $\sigma(y)$ is a quasi-best match for $x \in L$ (with respect to $(T,\sigma)$ and $u_T$) if both conditions (i) and (ii) are satisfied: (i) $y\,$ is a best match of $x$ in $(T,\sigma)$, (ii) $lca_T(x,y) \preceq u_T(x,\sigma(y))$. The digraph $qBMG(T,\sigma,u_T)$ is the vertex-colored digraph $qBMG(T,\sigma,u_T)$ on the vertex set $L$ whose edges are defined by the quasi-best matches. A vertex-colored digraph $(\dG,\sigma)$ with vertex set $L$ is a {\emph{$|S|$-colored quasi-best match graph (|S|-qBMG)}} if there is a leaf-colored tree $(T,\sigma)$ together with a truncation map $u_T$ on $(T,\sigma)$ such that $(\dG,\sigma) = qBMG(T,\sigma,u_T)$. For general results on quasi-best match graphs, the reader is referred to~\cite{korchmaros2023quasi}.

\section{Forbidden induced path graphs of underlying undirected 2-qBMGs}
\label{secindu}
The main result in the paper which establishes a new property of 2-qBMGs is given in the following theorem. 
\begin{theorem}\label{thm:P6free}
The underlying undirected graph of a 2-qBMG is $P_6$-free.
\end{theorem}
\begin{proof}
Let $\overrightarrow{G}$ denote a 2-qBMG with at least six vertices. Assume on the contrary that its underlying undirected graph $G$ has an induced subgraph $G_6$ on six vertices $v_1,v_2,v_3,v_4,v_5,v_6$ such that $v_1v_2v_3v_4v_5v_6$ is a $P_6$ path-graph: Then $G_6$ contains no chord other than $v_iv_{i+1}$ for $i=1,\ldots,5$. Four cases arise according to the possible patterns of the neighborhood of $v_2$ in $\dG$.

(i): $v_1v_2,v_3v_2\in E(\dG)$. Then $v_3v_4\in E(\dG),$ otherwise $[v_4,v_3,v_2,v_1]$ is an (N1)-quadruple.  
If $v_4v_5\in E(\dG)$ then $v_6v_5\in E(\dG)$, otherwise $v_3v_4v_5v_6$ violates (N2). On the other hand, $v_4v_5,v_6v_5\in E(\dG)$ violates (N1) as  $[v_3,v_4,v_5,v_6]$ is an (N1)-configuration. Hence $v_5v_4\in E(\dG)$. If $v_6v_5\in E(\dG)$ then $[v_6,v_5,v_4,v_3]$ is an (N1)-configuration. We are left with the case $v_1v_2,v_2v_3,v_3v_4,v_5v_4,v_5v_6\in E(\dG)$, as shown in Figure~\ref{fig:P6free}.

\begin{figure}[ht]
\centering
\scalebox{0.1}
{\includegraphics{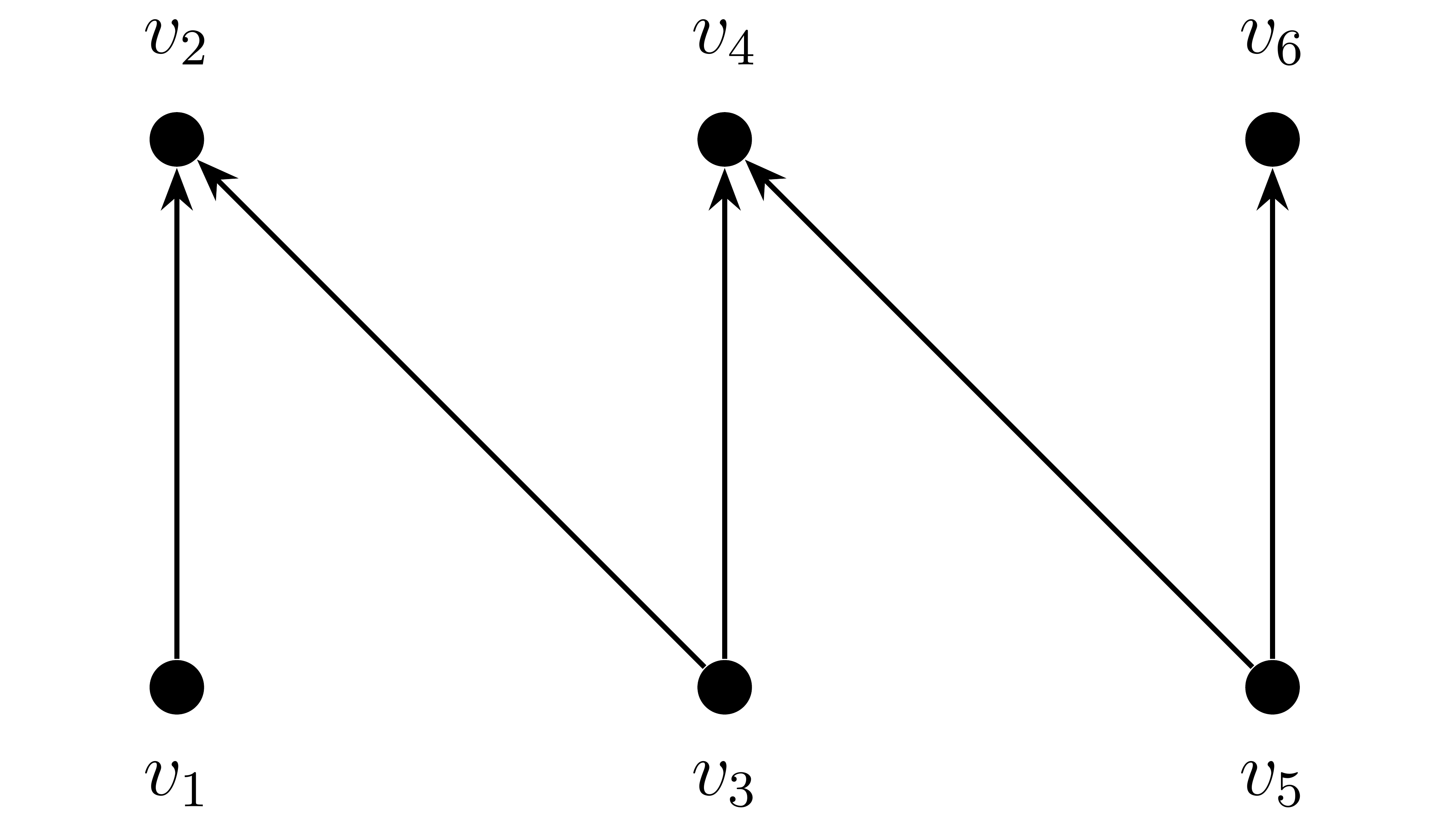}}
\caption{Case (i): all edges of $\dG$ up to possible symmetric edges.}
\label{fig:P6free}
\end{figure}

\noindent
In this case, (N3) does not hold for $v_3$ and $v_5$ since $v_4\in N^+(v_3)\cap N^+(v_5)$ whereas $v_6\in N^+(v_5)\setminus N^+(v_3)$ and $v_2\in N^+(v_3)\setminus N^+(v_5)$. 

(ii): $v_1v_2,v_2v_3\in E(\dG)$. Then $v_4v_3\in E(\dG)$, otherwise $v_1v_2v_3v_4$ violates (N2). On the other hand, $v_4v_3\in E(\dG)$ yields that $[v_1,v_2,v_3,v_4]$ is an (N1)-configuration, a contradiction. 

(iii): $v_2v_1,v_2v_3\in E(\dG)$. Suppose $v_4v_3\in E(\dG)$. Then (N3) yields $v_5v_4\in E(\dG)$ since  $v_3\in N^+(v_2)\cap N^+(v_4)$ and $v_1\in N^+(v_2)\setminus N^+(v_4)$. On the other hand, if  $v_5v_4\in E(\dG)$ then 
$[v_5,v_4,v_3,v_2]$ is an (N1)-configuration, a contradiction. Hence $v_3v_4\in E(\dG)$.
If $v_4v_5\in E(\dG)$ then $[v_2,v_3,v_4,v_5]$ is an (N1)-configuration, a  contradiction.
On the other hand, if $v_5v_4\in E(\dG)$ then $v_2v_3v_4v_5$
with $v_2 v_5\notin E(\dG)$ which contradicts (N2).

(iv): $v_2v_1,v_3v_2\in E(\dG)$. Then $v_3v_4\in E(\dG)$, otherwise $v_4v_3v_2v_1$ together with  $v_4v_1\notin E(\dG)$ violate (N2). Suppose that $v_5v_4\in E(\dG)$. Then 
(N3) yields $v_6v_5\in  E(\dG)$, since $v_4\in N^+(v_3)\cap N^+(v_5)$ and $v_2\in N^+(v_3)\setminus N^+(v_5)$. However, $v_6v_5\in  E(\dG)$ contradicts (N1), as it yields that   $[v_6,v_5,v_4,v_3]$ is an (N1)-configuration. Therefore we have $v_4v_5\in E(\dG)$. If   $v_6v_5\in E(\dG)$ then $[v_3,v_4,v_5,v_6]$ is an (N1)-configuration, a contradiction. On the other hand, if $v_5v_6\in E(\dG)$, $v_3v_4v_5v_6$ with $v_3v_6\notin E(\dG)$ violates (N2).
\end{proof}

Theorem~\ref{thm:P6free} implies that the underlying undirected graph of any 2-qBMG is $P_k$ free for $k\ge 6$, and also poses the problem of $P_k$-freeness for $3\le k\le 5$.  

\begin{theorem}\label{thm:P5} 
There exist 2-qBMGs on five vertices which are not $P_5$-free. More precisely,           
there are exactly six non-isomorphic 2-qBMGs on five vertices $\{v_1,v_2,v_3,v_4,v_5\}$ whose underlying undirected graph is a $P_5$ path-graph with edge-sets are:

\noindent
$E(\overrightarrow{P}_5^{(a)})=\{v_1v_2,v_3v_2,v_3v_4,v_4v_5\},$\\
$E(\overrightarrow{P}_5^{(b)})=\{v_1v_2,v_3v_2,v_3v_4,v_5v_4\},$\\
$E(\overrightarrow{P}_5^{(c)})=\{v_2v_1,v_3v_2,v_3v_4,v_4v_5\},$\\
$E(\overrightarrow{P}_5^{(a1)})= \{v_1v_2,v_2v_1,v_3v_2,v_3v_4,v_4v_5\}, $\\
$E(\overrightarrow{P}_5^{(b1)})=\{v_1v_2,v_2v_1,v_3v_2,v_3v_4,v_5v_4\},$\\
$E(\overrightarrow{P}_5^{(ab)})=\{v_1v_2,v_2v_1,v_3v_2,v_3v_4,v_4v_5,v_5v_4\}.$
\end{theorem}
\begin{proof} 
Let $\dG$ be a 2-qBMG on five vertices such that its underlying undirected graph $G$ has five pairwise distinct vertices, say $v_1,v_2,v_3,v_4,v_5$ where $v_1v_2v_3v_4v_5$ is a path containing no chord other than $v_iv_{i+1}$ for $i=1,\ldots,4$. The arguments in the proof of Theorem~\ref{thm:P6free} show that no 2-qBMG satisfies either $v_1v_2,v_2v_3\in E(\dG)$ or $v_2v_1,v_2v_3\in E(\dG)$. Therefore, two cases arise only according to the possible patterns of the neighborhood of $v_2$.

(i): $v_1v_2,v_3v_2\in E(\dG)$. The arguments in proof of Theorem~\ref{thm:P6free} show that $v_4v_3\notin E(\dG)$, and hence $v_3v_4\in E(\dG)$. Moreover either $v_4v_5 \in E(\dG)$ or $v_5v_4 \in E(\dG)$, and the arising digraphs are $\overrightarrow{P}_5^{(a)}$ and $\overrightarrow{P}_5^{(b)}$, respectively.
Adding edges with consecutive endpoints to $E(\overrightarrow{P}_5^{(a)})$ or $E(\overrightarrow{P}_5^{(b)})$ provide more non-isomorphic 2-qBMGs: If we add $v_2v_1$ creating a symmetric edge arises, we obtain two more non-isomorphic digraphs, named $\overrightarrow{P}_5^{(a1)}$ and $\overrightarrow{P}_5^{(b1)}$, respectively. Adding $v_5v_4$ to $E(\overrightarrow{P}_5^{(a)})$ or $v_4v_5$ to $E(\overrightarrow{P}_5^{(b)})$, the same 2-qBMG arises. Actually, it is isomorphic to $\overrightarrow{P}_5^{(b1)}$ by the map $\mu$ fixing $v_3$ and swapping $v_2$ with $v_4$, and $v_1$ with $v_5$. If we add $v_5v_4$ to $\overrightarrow{P}_5^{(a1)}$, or $v_4v_5$ to $\overrightarrow{P}_5^{(b1)}$, we obtain $\overrightarrow{P}_5^{(ab)}$.
On the other hand, the above method does not work with $v_2v_3$ or $v_4v_3$, since the arising digraph would not satisfy (N2).

(ii): $v_2v_1,v_3v_2\in E(\dG)$. The arguments in the proof of Theorem~\ref{thm:P6free} can also be used to show that $v_4v_3\notin E(\dG)$, whence $v_3v_4\in E(\dG)$ follows. Therefore, either $v_4v_5 \in E(\dG)$ or $v_5v_4 \in E(\dG)$. In the latter case, the digraph is isomorphic to $\overrightarrow{P}_5^{(a)}$ by the above map $\mu$. In the former case, $\overrightarrow{P}_5^{(c)}$ is obtained, which is not isomorphic to any of the 2-qBMGs already considered. Moreover, adding the edge $v_1v_2$ to $E(\overrightarrow{P}_5^{(c)})$ gives $\overrightarrow{P}_5^{(a1)}$. Therefore, no further non-isomorphic 2-qBMGs arise since adding either $v_2v_3$ or $v_4v_3$ would violate (N2).
\end{proof}

\begin{remark}
{\emph{In Theorem~\ref{thm:P5}, $\overrightarrow{P}_5^{(ab)}$ is a 2-BMG explained by the phylogenetic tree $(T,\sigma_p)$ in Figure~\ref{fig:P5free}, where $\sigma_p$ is the leaf-coloring defined according to the parity of labels of the leaves. Moreover, the other five 2-qBMGs in Theorem~\ref{thm:P5} are explained by $(T,\sigma_p,\tau)$ where $\tau$ is an appropriate sink-truncation map depending on the digraph.}}
\begin{figure}[ht]
\centering
\scalebox{0.1}
{\includegraphics{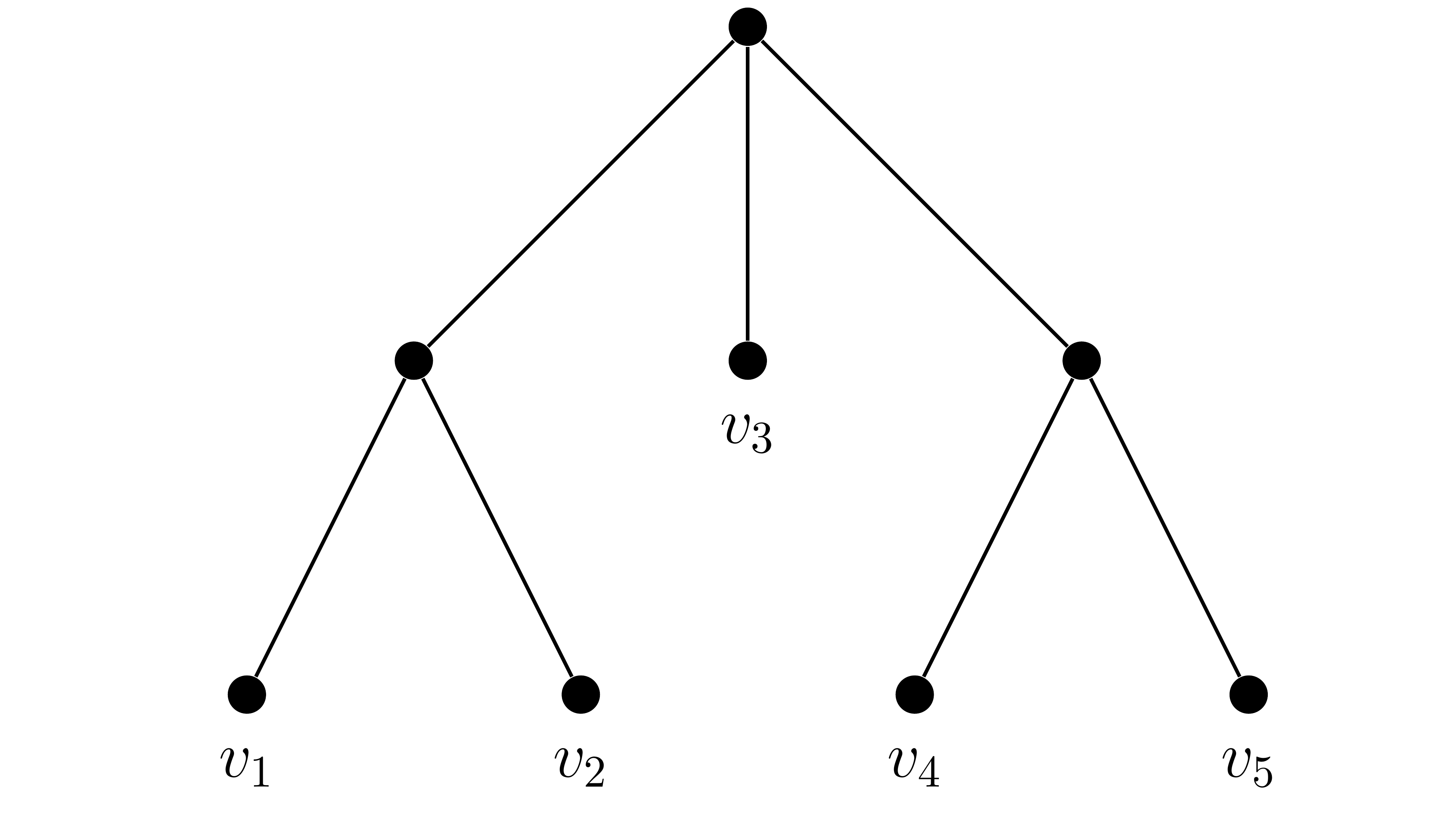}}
\caption{Tree topology explaining all digraphs in Theorem~\ref{thm:P5}. Leaves $v_1,v_3,v_5$ have the same color opposite to the color of leaves $v_2,v_4$.}
\label{fig:P5free}
\end{figure}
\end{remark}

\begin{example} {\emph{The 2-qBMG $\dG$ on $7$ vertices $\{v_1,v_2,\ldots,v_7\}$ and edge-set 
$$E(\dG)=\{v_5v_4,v_2v_1,v_3v_4,v_3v_2,v_4v_7,v_1v_6,v_3v_6,v_6v_1,v_7v_4\}$$ is an example of a larger 2-qBMG whose induced subgraph on $\{v_1,v_2,v_3,v_4,v_5\}$ coincides with $\overrightarrow{P}_5^{(a1)}$.}}
\end{example}
A case-by-case analysis similar but simpler to those in the proof of Theorem~\ref{thm:P5} gives the following result.

\begin{theorem} \label{thm:P4-cases}
There are exactly four non-isomorphic 2-qBMGs on four vertices $\{v_1,v_2,v_3,v_4\}$ whose underlying undirected graph is a $P_4$ path-graph: 

\noindent
$E(\overrightarrow{P}_4^{(1)})=\{v_1v_2,v_1v_4,v_2v_3\},$\\
$E(\overrightarrow{P}_4^{(2)})=\{v_1v_2,v_1v_4,v_3v_4\}, $\\
$E(\overrightarrow{P}_4^{(3)})=\{v_1v_2,v_1v_4,v_2v_1,v_2v_3\},$\\
$E(\overrightarrow{P}_4^{(4)})=\{v_1v_2,v_2v_1,v_4v_1,v_4v_3\}. $
\end{theorem}
In Theorem~\ref{thm:P4-cases}, all the non-isomorphic 2-qBMGs on four vertices contain a sink. Therefore, the undirected underlying graph of a 2BMG is $P_4$-free since 2BMGs are sink-free 2-qBMG. 
\begin{corollary}\label{cor:un2BMG-cograph}
The underlying undirected graph of a 2BMG is a cograph.
\end{corollary}
The $P_3$-freeness problem is solved in the following proposition. 
\begin{proposition}\label{prop:3vertices}
Every bipartite digraph on three vertices $\{v_1,v_2,v_3\}$ is a 2-qBMG.
\end{proposition}
\begin{proof}
It is straightforward to check that both (N1) and (N2) hold trivially for every bipartite digraph $\dG$ on three vertices. For (N3), it also holds trivially when $|E(\dG)|=1$. On the other hand, assign a vertex-coloring on $V(\dG)$ such that $v_1$ and $v_3$ have the same color. The only case when two vertices have at least one common out-neighbor is when $N^+(v_1)=N^+(v_3)=\{v_2\}$. Hence, (N3) holds also in the case when $|E(\dG)|>1$.  
\end{proof}
A straightforward consequence of Proposition~\ref{prop:3vertices} is a characterization of the 2-qBMGs on three vertices whose underlying undirected graph is a $P_3$ path-graph.
\begin{corollary}
The non-isomorphic 2-qBMGs on three vertices whose underlying undirected graph is a $P_3$ path-graph are all bipartite digraphs with at least two edges except for a symmetric edge.
\end{corollary}

The same setup and arguments from the proof of Theorem~\ref{thm:P6free} can be used to deal with induced $6$-cycles in 2-qBMGs since the proof of Theorem~\ref{thm:P6free} involves neither $v_1v_6$ nor $v_6v_1$. Therefore, the following result holds.
\begin{theorem}
\label{thm:C6free} The underlying undirected graph of any 2-qBMG is $C_6$-free.
\end{theorem}
\begin{remark}
\label{remA13052024}
{\emph{ 2-qBMGs are bipartite digraphs; therefore, they cannot induce cycles of an odd length. 
In particular, they are $C_5$-free and $C_3$-free. On the other hand, 2-qBMGs with paths and cycles of length $\ell<5$ exist. For $\ell=4$, consider the digraph $\dG$ on $4$ vertices $\{v_1,v_2,v_3,v_4\}$ and edge-set $E(\dG)=\{v_2v_1,v_2v_3,v_3v_4\}$. The underlying graph of $\dG$ is a path of length $4$, and $\dG$ is a 2-qBMG; indeed it is explained by $(T,\sigma_p, \tau)$, where $T$ is represented in Figure~\ref{fig:P4-P3}(a), $\sigma_p$ is the leaf-coloring map defined by the parity of the leaf labels and $\tau$ is the sink-truncation maps of $G$. For $l=3$, consider the digraph $\dG$ on $3$ vertices $\{v_1,v_2,v_3\}$ and edge-set$\dG=\{v_2v_1,v_2v_3\}$ is a 2-qBMG explained by $(T,\sigma_p, \tau)$, where $T$ is in Figure~\ref{fig:P4-P3}(b). The underlying graph of $\dG$ is a path of length $3$.}}

\begin{figure}[ht]
\centering
\scalebox{0.1}
{\includegraphics{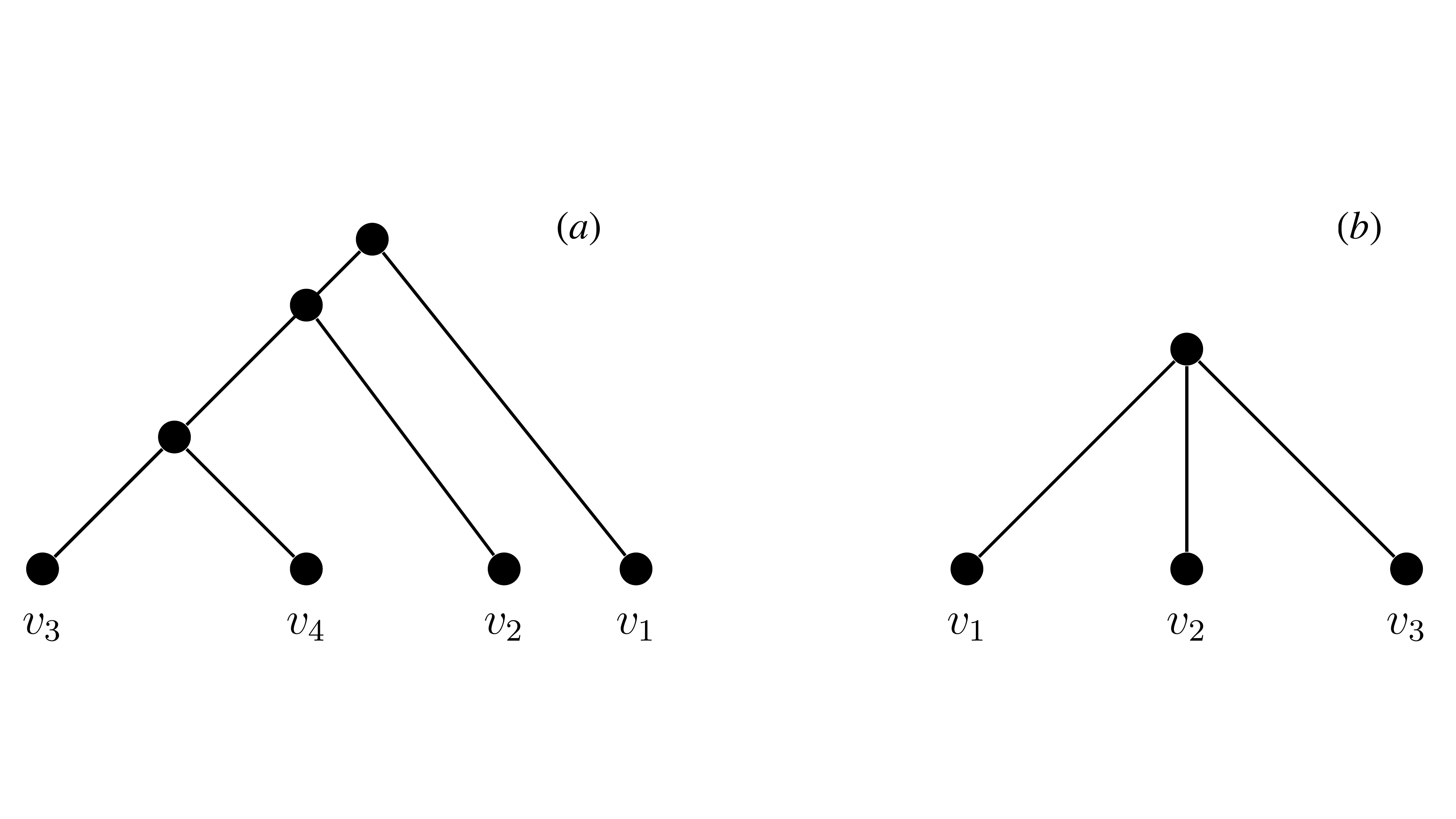}}
\caption{(a) Tree topology explaining a 2-qBMG with an induced $P_4$. (b) Tree topology explaining a 2-qBMG with an induced $P_3$.}
\label{fig:P4-P3}
\end{figure}
\end{remark}
The arguments used in the proof of Theorem~\ref{thm:P5} can be adapted to prove the following theorem. 

\begin{theorem} \label{thm:C4-cases}
There are exactly 10 non-isomorphic 2-qBMGs on four vertices $\{v_1,v_2,v_3,v_4\}$ whose underlying undirected graph is a $C_4$. They are the following ten digraphs: 

\noindent
$E(\overrightarrow{P}_4^{(1)})=\{v_1v_2,v_1v_4,v_2v_1,v_2v_3,v_3v_4,v_4v_3\}$,\\
$E(\overrightarrow{P}_4^{(2)})=\{v_1v_2,v_1v_4,v_2v_3,v_3v_2,v_3v_4,v_4v_3\},$\\
$E(\overrightarrow{P}_4^{(3)})=\{v_1v_2,v_1v_4,v_3v_2,v_3v_4\},$\\
$E(\overrightarrow{P}_4^{(4)})=\{v_1v_2,v_1v_4,v_3v_2,v_4v_3\},$\\
$E(\overrightarrow{P}_4^{(5)})=\{v_1v_2,v_2v_1,v_3v_2,v_3v_4,v_4v_1\},$\\
$E(\overrightarrow{P}_4^{(6)})=\{v_1v_2,v_1v_4,v_3v_2,v_4v_1,v_4v_3\}, $\\
$E(\overrightarrow{P}_4^{(7)})= \{v_1v_2,v_1v_4,v_2v_3,v_4v_3\},$\\
$E(\overrightarrow{P}_4^{(8)})=\{v_1v_2,v_1v_4,v_3v_2,v_3v_4,v_4v_3\},$\\
$E(\overrightarrow{P}_4^{(9)})=\{v_1v_2,v_1v_4,v_2v_1,v_2v_3,v_3v_2,v_3v_4\},$\\
$E(\overrightarrow{P}_4^{(10)})=\{v_1v_2,v_1v_4,v_3v_2,v_3v_4,v_2v_1,v_2v_3,v_4v_1,v_4v_3\}.$\\
\end{theorem}

\section{Dominating sets in 2-qBMGs}
As pointed out in the introduction, Theorem~\ref{thm:P6free} and Theorem~\ref{thm:C6free}, together with previous results, give new contributions to the current studies on the structure of 2-qBMGs. This section is focused on dominating biclique sets.

Recall that a connected 2-qBMG is of type (A) if its undirected underlying graph has a vertex decomposition into a biclique and a stable set. It may be observed that such a biclique is necessarily a dominating set.
\begin{example} {\emph{The digraph on $10$ vertices $\{v_1,v_2, \ldots, v_{10}\}$ with edge-set 
\begin{align*}
E(\dG)=\{&v_1v_5, v_1v_6, v_1v_7, v_1v_8, v_5v_2, v_6v_2, v_2v_7, v_2v_8,\\
       &v_5v_3, v_6v_3, v_3v_7, v_3v_8, v_5v_4, v_6v_4, v_7v_4, v_4v_8, \\
       &v_5v_9, v_1v_{10}, v_2v_{10}\}
\end{align*}
is a 2-qBMG of type (A) whose underlying undirected graph $G$ has an induced dominating subgraph with vertex set $\{v_1,v_2,v_3,v_4,v_5,v_6,v_7,v_8\}$ and stable set $\{v_9,v_{10}\}$.}}
\end{example}

\begin{theorem}
\label{thm:connected-typeA}
Every connected 2-qBMG has a vertex decomposition into connected 2-qBMGs of type (A).
\end{theorem}
\begin{proof}
Let $\overrightarrow{G}\,$ be a connected 2-qBMG not of type (A).
Let $G$ be the underlying undirected graph of $\overrightarrow{G}$, where $U$ and
$W$ are its color classes. 
As mentioned in the introduction, from~\cite[Theorem 1]{liu1994dominating}, $G$ contains a dominating biclique $\Delta$ with vertex set $T\cup Z$ where $T\subseteq U$ and $Z\subseteq W$. Let $\overrightarrow{\Delta}$ be the directed subgraph of $\dG$ induced by $\Delta$. 
Let $S$ be the subset of $V(G)$ consisting of all vertices $v\in V(G)$ for which $N^+(v)\cup N^-(v)$ is contained in $T\cup Z$. Let $\overrightarrow{\Sigma}$ be the induced subgraph of $\overrightarrow{G}$ on $T\cup Z\cup S$. Since $\Delta$ is a dominating set of $G$, $\overrightarrow{\Sigma}$ is connected without isolated vertex and is a connected 2-qBMG. The underlying undirected graph $\Sigma$ of $\overrightarrow{\Sigma}$ is a $K\oplus S$ graph, where $K=T\cup Z$.

Furthermore, let $\overrightarrow{G}_1$ be the induced subgraph of $\overrightarrow{G}$ on $V(G)\setminus (T\cup Z\cup S)$. By construction, no vertex of $\overrightarrow{G}_1$ is isolated, although some vertex of $\overrightarrow{G}$ may become a sink in $\overrightarrow{G}_1$, and non-equivalent vertices $\overrightarrow{G}$ may become equivalent in $\overrightarrow{G}_1$. If $G_1$ is not a connected 2-qBMG of type (A), we can repeat the above construction to define a finite sequence $\Sigma_1,G_2, \Sigma_2\, \ldots,\Sigma_k$ such that $\overrightarrow{\Sigma}, \overrightarrow{\Sigma}_1,\overrightarrow{\Sigma}_2,\ldots, \overrightarrow{\Sigma}_k$ form a vertex partition of $\overrightarrow{G}$ into connected 2-qBMGs of type (A).

If some of the members of the subsequence $G_1,G_2,\ldots,G_{k-1}$ decomposition is not connected, say $G_m$, applying the above argument on each of the connected components of $G_m$ gives the claimed partition.
\end{proof} 

\begin{proposition}
    \label{prop:delta_bmg} 
Let $\dG$ be a 2-qBMG satisfying at least one of conditions ($*$) and ($**$).
Fix an orientation $\overrightarrow{\Gamma}$ of $\dG$.
Let $\overrightarrow{\Delta}$ be an induced subgraph of $\,\,\overrightarrow{\Gamma}$ such that
the underlying undirected graph $\Delta$ is a biclique of $G$. If $\dG$ has no symmetric edges in $\overrightarrow{\Delta}$ then $\overrightarrow{\Delta}$ is a 2-qBMG.
\end{proposition}
\begin{proof} 
Let $n$ be the number of the vertices of $\dG$ (and of $\overrightarrow{\Gamma}$).  As pointed out in Section~\ref{bg}, if either ($*$) or ($**$) holds then 
$\overrightarrow{\Gamma}$ has a topological ordering. Therefore, the vertices of $\overrightarrow{\Gamma}$  can be labelled with $\{v_1,v_2,\ldots, v_n\}$ such that $v_iv_j\in E(\overrightarrow{\Gamma})$ implies $i<j$ for $1\le i,j \le n$.
In particular, for $v_i,v_j\in V(\Delta)$ of different color, either $v_iv_j$ or $v_jv_i$ is an edge of $\overrightarrow{\Delta}$. Therefore, (N1) trivially holds for $\overrightarrow{\Delta}$ while (N2) follows from the transitivity of the ordering. To show (N3), take $v_i,v_j\in V(\Delta)$ of the same color such that $v_i$ and $v_j$ have a common out-neighbour $w\in \overrightarrow{\Delta}$. Then $w$ is a common out-neighbour of $v_i$ and $v_j$ in $\dG$. From (N3) applied to $\dG$, all out-neighbours of $v_i$ in $\dG$ are also out-neighbours of $v_j$ (or, vice versa). Now take any vertex $v$ such that $v_iv\in E(\overrightarrow{\Delta})$. Then $v_iv\in E(\dG)$, and hence $v_jv\in E(\dG)$. Since $v_jv$ is not a symmetric edge in $\dG$, this yields  $v_jv\in E(\overrightarrow{\Gamma})$ whence $v_jv\in E(\overrightarrow{\Delta}$) follows. Thus $\overrightarrow{\Delta}$ is a 2-qBMG.
\end{proof}
\begin{theorem}
Let $\dG$ be a connected 2-qBMG satisfying ($*$). Then $\dG$ has an even-odd subgraph $(\mathfrak{A}, \mathfrak{O})$ such that the underlying undirected subgraph of $(\mathfrak{A}, \mathfrak{O})$ is a dominating biclique of $G$.
\end{theorem}
\begin{proof}  Fix an orientation $\overrightarrow{\Gamma}$ of $\overrightarrow{G}$ by keeping the same vertex set but retaining exactly one edge from each symmetric edge. Let $G$ be the underlying undirected graph of $\overrightarrow{G}$. Then $G$ coincides with the underlying undirected graph of $\overrightarrow{\Gamma}$. As mentioned in Introduction, from~\cite[Theorem 1]{liu1994dominating}, $G$ contains a dominating biclique $\Delta$ with vertex set $T\cup Z$ where $T\subseteq U$ and $Z\subseteq W$  where $U$ and $W$ are the bipartion classes of $G$. Let $\overrightarrow{\Delta}$ be the directed subgraph on vertex-set $T\cup Z$ where $v_iv_j\in E(\overrightarrow{\Delta})$ if and only if 
$v_iv_j$ is an edge of $\overrightarrow{\Gamma}$.
It should be noticed that $\overrightarrow{\Delta}$ may not be the induced subgraph of $\overrightarrow{G}$ on the vertex-set $T\cup Z$. This occurs indeed when $\overrightarrow{G}$ contains a symmetric edge with both endpoints in $T\cup Z$. Nevertheless, $\overrightarrow{\Delta}$ is bitournament. Since $\overrightarrow{\Delta}$ is bitransitive, it is an even-odd digraph, as recalled in Section~\ref{bg}.
\end{proof}

\section{Future directions} 
The $P_6$-and $C_6$-freeness in the undirected underlying 2-qBMGs raises the problem of finding more such forbidden subgraphs on six (or eight) vertices. In this direction, the domino graph on six vertices (i.e. the Cartesian product $P_2 \times P_3$) and its generalization on eight vertices, the long-domino or ladder graph, appear worth investigating.      

In~\cite{quaddoura2024bipartite}, bipartite $P_6$ and $C_6$-free graphs are characterized as those graphs where every connected subgraph is of type $K\oplus S$. This gives rise to a linear time algorithm for recognizing if an undirected graph is $P_6$ and $C_6$-free by tracing all $K\oplus S$ decompositions. It would be interesting to investigate whether a similar approach can be used to perform a linear time algorithm to recognize if a digraph is a 2-qBMG. The first step is to check whether the converse of ~\cite[Theorem 2]{quaddoura2024bipartite} holds.

The s-dim of a graph $G$ is the minimum number of bicliques needed to cover the edge-set of $G$. Computing the s-dim (biclique cover problem) is NP-complete for bipartite graphs~\cite{dawande2001biclique}; however, this problem is linear for some classes of graphs, such as bipartite domino-free graphs~\cite{amilhastre1998complexity}. Investigating the complexity of the biclique cover problem for the family of underlying undirected graphs of 2-qBMGs is also an interesting topic for future investigation.

For two disjoint subsets $W$ and $T$ of vertices, $W$ \emph{2-dominates} $T$ if every vertex of $T$ is
 adjacent to at least two vertices of $W$. A partition $\{V_1,V_2, \ldots,V_n\}$ of vertices in $V(G)$ into $n$ parts is a \emph{2-transitive partition} of size $n$ if $V_i$ 2-dominates for all $1 \leq i < j \leq n$. Finding a 2-transitivity partition of maximum order is NP-complete for bipartite graphs~\cite{paul2023algorithmic}. Studying the complexity of this problem for the family of underlying undirected graphs of 2-qBMGs is also an interesting topic for future investigation.

\bibliography{refs}

\end{document}